\documentclass[11pt,a4paper]{article}
\usepackage[T1]{fontenc}
\usepackage[utf8]{inputenc}
\usepackage{lmodern}
\usepackage[margin=28mm]{geometry}
\usepackage{amsmath,amssymb,amsthm}
\usepackage[hidelinks]{hyperref}
\newtheorem{theorem}{Theorem}[section]
\newtheorem{lemma}[theorem]{Lemma}
\newtheorem{proposition}[theorem]{Proposition}
\newtheorem{corollary}[theorem]{Corollary}
\theoremstyle{definition}
\newtheorem{definition}[theorem]{Definition}

\newcommand{\C}{\mathcal C}
\newcommand{\Hom}{\operatorname{Hom}}
\newcommand{\Lie}{\operatorname{Lie}}
\newcommand{\Const}{\operatorname{Const}}
\newcommand{\LL}{\widehat{\mathbb L}}
\newcommand{\id}{\operatorname{id}}
\title{Continuous Comparison of Free Simplicial\protect\\
Prounipotent and Pro-\(p\) Resolutions}
\author{Andrey Mikhovich}
\date{}
\begin{document}
\maketitle
\begin{abstract}
We prove a continuous comparison theorem for based free simplicial
resolutions of prounipotent groups in characteristic zero and of pro-$p$
groups, allowing arbitrary rank. Continuous linear splittings and the
universal property of completed free Lie algebras give projective lifting
in characteristic zero; in the pro-$p$ case we use projectivity of free
pro-$p$ groups. Relative lifting through full matching objects gives
comparison maps, and a simplicial cylinder gives homotopy uniqueness.
The result applies to completed resolutions arising from subcontractible
presentations.
\end{abstract}

\section{The question and the category}

In our preceding work \cite{Mikhovich}, the completed simplicial group
associated with a finite subcontractible presentation is compared with a
constant free group after its positive homotopy groups have been shown to
vanish. This raises the question of how Keune's comparison
argument \cite{Keune} extends to free simplicial prounipotent resolutions,
especially when higher degrees have infinite rank. The issue is continuity:
arbitrary lifts of infinitely many free generators need not form a
convergent family. The same question arises for pro-$p$ resolutions in the
simplicial profinite framework of Quillen \cite{Quillen}.

We write $\C$ for either of the following categories: prounipotent
group schemes over a fixed field $k$ of characteristic zero, with
group-scheme morphisms; or pro-$p$ groups, with continuous homomorphisms.
Products, kernels and coproducts are taken in $\C$.
In the first category we use the equivalence with inverse limits of
finite-dimensional nilpotent $k$-Lie algebras
\cite[Appendix~A]{QuillenRHT}, \cite[Section~2]{HainMatsumoto}.
Their topology is the inverse-limit topology with discrete
finite-dimensional quotients, not an analytic topology on $k$.
An epimorphism below means a surjective morphism (faithfully flat in the
group-scheme terminology); on Lie algebras it is a continuous surjection.
Throughout, capital letters $U,V,\ldots$ denote objects of $\C$; in the
prounipotent case we write $\mathfrak u=\Lie(U)$, $\mathfrak v=\Lie(V)$,
and similarly for other objects.

For a set $S$, let $S^+=S\sqcup\{*\}$ have isolated points $s\in S$
and cofinite neighbourhoods of $*$. Write $F_p(S)$ for the free pro-$p$
group on a family indexed by $S$ converging to $1$. Thus, for every
pro-$p$ group $U$, continuous homomorphisms $F_p(S)\to U$ correspond to
continuous pointed maps $S^+\to U$. This convention applies also to infinite $S$.
In characteristic zero put
\[
 E_S=k^S,\qquad F_k(S)=\exp\LL(E_S),
\]
where $k^S$ has its product topology and $\LL$ is defined below.
The canonical generators are $\exp(e_s)$, where $e_s$ is the coordinate
vector. Their images in every finite-dimensional unipotent quotient are
$1$ for all but finitely many $s$. More generally, such convergent
assignments of generators have the expected universal property, by
taking logarithms. We use $F(S)$ for either construction.

For a finite partition $S=S_1\sqcup\cdots\sqcup S_t$, these conventions give
\begin{equation}\label{eq:partition}
 F(S)\cong\coprod_{j=1}^t F(S_j).
\end{equation}
Indeed, a convergent assignment on $S$ is precisely a convergent
assignment on each of its finitely many parts. This statement concerns
coproducts in $\C$, not abstract free products of underlying groups.

\section{Continuous lifting}

\begin{lemma}\label{lem:linear}
A continuous surjective linear map $q:V\to W$ between linearly compact
$k$-spaces admits a continuous linear section.
\end{lemma}
\begin{proof}
Continuous duality identifies linearly compact spaces contravariantly with
discrete vector spaces. In particular $V\cong(V^\vee)^*$, where
$V^\vee=\Hom_{\rm cts}(V,k)$ and $(-)^*$ is the full dual with its
linearly compact topology. The map $q^\vee:W^\vee\hookrightarrow V^\vee$
is injective. Extend a basis of its image to a basis of $V^\vee$, and
choose a linear retraction $r:V^\vee\to W^\vee$.
Then $s=r^*:W\to V$ is continuous and $qs=\id_W$.
\end{proof}

For a vector space $V$, write $\mathbb L(V)$ for the ordinary free Lie
algebra on $V$.  For any Lie algebra $\mathfrak g$ its lower central series is
\[
 \gamma_1\mathfrak g=\mathfrak g,\qquad
 \gamma_{n+1}\mathfrak g=[\mathfrak g,\gamma_n\mathfrak g]
 \quad(n\geq1).
\]
Thus $\mathbb L(V)/\gamma_{c+1}\mathbb L(V)$ is the free nilpotent Lie
algebra of class at most $c$ on $V$.

For a linearly compact space $E$ define
\begin{equation}\label{eq:free-lie}
 \LL(E)=\varprojlim_{W,c}
 \mathbb L(E/W)/\gamma_{c+1}\mathbb L(E/W),
\end{equation}
where $W$ runs through open subspaces of finite codimension and $c\geq1$.
The inverse system is ordered by refinement: $(W',c')$ maps to $(W,c)$
whenever $W'\subseteq W$ and $c'\geq c$, via the evident quotient map.
Every displayed factor is a finite-dimensional nilpotent Lie algebra.
The canonical map $E\to\LL(E)$ is continuous, and its image
topologically generates $\LL(E)$.

\begin{lemma}\label{lem:universal}
For every pronilpotent Lie algebra $\mathfrak u$, restriction to $E$ gives
a bijection
\[
 \Hom_{\rm cts,Lie}(\LL(E),\mathfrak u)
 \cong \Hom_{\rm cts,lin}(E,\mathfrak u).
\]
\end{lemma}
\begin{proof}
Write $\mathfrak u=\varprojlim_J\mathfrak u/J$, with finite-dimensional
nilpotent quotients, and let $a:E\to\mathfrak u$ be continuous and linear.
For each $J$, the composite $a_J:E\to\mathfrak u/J$ factors through
$E/W$ for an open finite-codimensional subspace $W\subseteq\ker a_J$.
If $c$ is at least the nilpotency class of $\mathfrak u/J$, the free
nilpotent universal property gives
\[
 \mathbb L(E/W)/\gamma_{c+1}\mathbb L(E/W)\longrightarrow\mathfrak u/J.
\]
Composition with the projection from $\LL(E)$ defines a continuous
Lie map $\Phi_J$ extending $a_J$. Two choices $(W_1,c_1)$ and
$(W_2,c_2)$ have the common refinement
$(W_1\cap W_2,\max\{c_1,c_2\})$; uniqueness on that free nilpotent
quotient shows that they give the same $\Phi_J$. The same argument for
$J'\subseteq J$ proves compatibility with the quotient maps.
Thus $(\Phi_J)_J$ defines a continuous Lie map
$\Phi:\LL(E)\to\mathfrak u$ extending $a$.

The Lie subalgebra generated by $E$ is dense in $\LL(E)$.
Two continuous extensions therefore agree in every Hausdorff quotient
$\mathfrak u/J$, and hence in $\mathfrak u$.
\end{proof}

\begin{proposition}\label{prop:projective}
Let $q:U\twoheadrightarrow V$ be an epimorphism in $\C$.
Every morphism $f:F(S)\to V$ admits a morphism
$\widetilde f:F(S)\to U$ with $q\widetilde f=f$.
There is no restriction on the cardinality of $S$.
\end{proposition}
\begin{proof}
In characteristic zero the assertion holds more generally with
$F(S)$ replaced by $\exp\LL(E)$ for any linearly compact $E$.
Let $q_*:\mathfrak u\twoheadrightarrow\mathfrak v$ and
$f_*:\LL(E)\to\mathfrak v$ be the corresponding Lie morphisms.
Choose the continuous linear section $s$ of $q_*$ furnished by
Lemma~\ref{lem:linear}. The continuous linear map
$s\circ f_*|_E$ extends by Lemma~\ref{lem:universal} to a continuous
Lie morphism $\widetilde f_*:\LL(E)\to\mathfrak u$.
The maps $q_*\widetilde f_*$ and $f_*$ agree on $E$, and hence agree
everywhere. Exponentiation proves the assertion. The section $s$ is
not required to preserve Lie brackets.

In the pro-$p$ category this is the classical projectivity theorem for
free pro-$p$ groups, valid in arbitrary rank
\cite[Section~7.7]{RibesZalesskii}. We recall how its finite lifting
criterion applies here. For a diagram $F_p(S)\to B\leftarrow A$ with
$A\twoheadrightarrow B$ a surjection of finite $p$-groups, the images
of all but finitely many generators in $B$ are $1$. Lift the finitely
many nonidentity images and send every remaining generator to $1$.
The resulting assignment is convergent and extends to a continuous
solution. The standard projectivity criterion for pro-$p$ groups says
that solving all such finite embedding problems is equivalent to lifting
along arbitrary epimorphisms of pro-$p$ groups
\cite[Section~7.6]{RibesZalesskii}. This yields the stated lifting.
\end{proof}

\begin{corollary}\label{cor:relative}
Suppose $B=A\amalg F(S)$ in $\C$, $q:U\twoheadrightarrow V$ is an
epimorphism, and $a:A\to U$, $f:B\to V$ satisfy $qa=f|_A$.
Then $f$ has a lift $\widetilde f:B\to U$ with $\widetilde f|_A=a$.
\end{corollary}
\begin{proof}
Lift $f|_{F(S)}$ by Proposition~\ref{prop:projective} and combine this
lift with $a$ using the coproduct universal property.
\end{proof}

\section{Free simplicial objects and relative extension}

An augmentation $X_\bullet\to G$ means an augmentation to
$\Const(G)$. Put
$Y_{-1}=H$ and regard an augmented simplicial object as a functor on the
augmented simplex category. For $n\geq0$ define the \emph{full matching
object}
\[
 M_nY=\varprojlim_{\substack{\delta:[m]\hookrightarrow[n]\\-1\leq m<n}}Y_m,
\]
where the limit is taken in $\C$ over all proper injective maps into
$[n]$. Thus $M_nY$ is the object of all proper faces of a formal
$n$--simplex which agree on their common subfaces. The natural matching
map
\[
 b_n:Y_n\longrightarrow M_nY
\]
is induced by the face maps (and, in degree zero, by the augmentation).
In particular
\[
 M_0Y=H,\qquad M_1Y=Y_0\times_HY_0,
\]
and, for $n\geq1$, the codimension-one faces determine the whole
matching object, giving the concrete description
\[
 M_nY=\{(y_0,\ldots,y_n)\in Y_{n-1}^{n+1}:
              d_i y_j=d_{j-1}y_i\text{ for }i<j\}.
\]
For $n=1$ the single compatibility equation means that the two elements
of $Y_0$ have the same image in $H$. For example,
\[
 M_2Y=\{(y_0,y_1,y_2)\in Y_1^3:
 d_0y_1=d_0y_0,\ d_0y_2=d_1y_0,\ d_1y_2=d_1y_1\}.
\]
Under the tuple description, $b_n=(d_0,\ldots,d_n)$. Thus surjectivity
of $b_n$ says precisely that every compatible \emph{full boundary}
admits an $n$--simplex filling it; this is stronger data than a single
horn-filling statement.

For compatibility with Keune's notation, put
\[
 Z_{-1}Y=H,\qquad Z_nY=M_{n+1}Y\quad(n\geq0),
\]
and write
\[
 d^{(-1)}=b_0:Y_0\longrightarrow H,\qquad
 d^{(n)}=b_{n+1}:Y_{n+1}\longrightarrow Z_nY\quad(n\geq0).
\]
Thus $Z_nY$ is the \emph{simplicial kernel}: it consists of compatible
full boundaries in degree $n+1$. We call $Y_\bullet\to H$ a
\emph{resolution} when every $d^{(n)}$, $n\geq-1$, is an epimorphism.

\begin{proposition}
\label{prop:keune-simplicial-kernel}
Let $Y_\bullet\to H$ be an augmented simplicial group. For every
$n\geq0$ there is a natural bijection of pointed sets
\[
 \widetilde\gamma_n:
 \pi_n(Y_\bullet\to H)\xrightarrow{\ \cong\ }
 Z_nY/\operatorname{im}d^{(n)}.
 \tag{SK}
\]
To specify both sides and their group structures, put
\[
 \begin{aligned}
 N_mY&=\bigcap_{i=0}^{m-1}\ker d_i &&(m\geq1),\\
 C_nY&=\bigcap_{i=0}^{n}\ker d_i &&(n\geq1),\\
 C_0Y&=\ker(Y_0\to H).
 \end{aligned}
\]
and write
\[
 B_nY=d_{n+1}(N_{n+1}Y),\qquad
 D_nY=\operatorname{im}d^{(n)}.
\]
Thus $\pi_n(Y_\bullet\to H)=C_nY/B_nY$ is the homotopy group
of the augmented Moore complex. The map in \textup{(SK)} is induced by
\[
 \gamma_n(c)=(1,\ldots,1,c)\in Z_nY,\qquad c\in C_nY.
\]
The quotient on the right of \textup{(SK)} means the set of left cosets
$zD_nY$, with distinguished element $D_nY$. It is equipped with the
group structure transported from $C_nY/B_nY$ along
$\widetilde\gamma_n$. Explicitly, every coset has a representative
$\gamma_n(c)$, and the multiplication is
\[
 \bigl(\gamma_n(c)D_nY\bigr)\star
 \bigl(\gamma_n(c')D_nY\bigr)
 =\gamma_n(cc')D_nY.
\]
With this multiplication, \textup{(SK)} is a natural isomorphism of
groups. No normality of $D_nY$ in $Z_nY$ is assumed, and the coset
projection from $Z_nY$ is not asserted to be a group homomorphism.
In the pro-$p$ and prounipotent cases, images are closed; the topology,
respectively the prounipotent group structure, on the right of
\textup{(SK)} is transported from the Moore quotient.
Consequently, $d^{(n)}$ is surjective exactly when the degree-$n$
augmented homotopy obstruction vanishes.
\end{proposition}

\begin{proof}
For the Moore-complex description and full-boundary bijection, see
Inassaridze \cite[Chapter~2, Definition~2.2, pp.~35--36, and
Proposition~2.14, pp.~41--42]{Inassaridze}. We give the coset argument
in Keune's simplicial-kernel notation \cite{Keune}.

The simplicial identities give $B_nY\subseteq C_nY$.
Moreover, $N_{n+1}Y$ is normal in $Y_{n+1}$, and $d_{n+1}$ is
surjective because $d_{n+1}s_n=\id$. Hence $B_nY$ is normal in
$Y_n$, and in particular in $C_nY$.
We claim that
\begin{equation}\label{eq:coset-comparison}
 \begin{aligned}
 \gamma_n(C_nY)\cap D_nY&=\gamma_n(B_nY),\\
 Z_nY&=\gamma_n(C_nY)D_nY.
 \end{aligned}
\end{equation}
Indeed, $\gamma_n(c)=d^{(n)}x$ means precisely that
$d_0x=\cdots=d_nx=1$ and $d_{n+1}x=c$. This is equivalent to
$x\in N_{n+1}Y$ with Moore boundary $c$, proving the first equality.

For the second equality, let $z=(z_0,\ldots,z_{n+1})\in Z_nY$.
Fill the horn consisting of its first $n+1$ faces: there is an
$x\in Y_{n+1}$ with $d_ix=z_i$ for $0\leq i\leq n$.
For completeness, such an $x$ is obtained by the finite recursion
\[
 \begin{aligned}
 x^{(0)}&=1,\qquad x=x^{(n+1)},\\
 x^{(i+1)}&=s_i\bigl(z_i(d_ix^{(i)})^{-1}\bigr)x^{(i)}
 \quad(0\leq i\leq n).
 \end{aligned}
\]
The matching equations ensure that each step preserves the earlier
faces. Put
\[
c=z_{n+1}(d_{n+1}x)^{-1}.
\]
The matching equations give $c\in C_nY$, including the augmentation
condition when $n=0$.
Thus
\[
 z\bigl(d^{(n)}x\bigr)^{-1}=\gamma_n(c),
\]
which proves the second equality in \eqref{eq:coset-comparison}.

It follows that $cB_nY\mapsto\gamma_n(c)D_nY$ is a well-defined
bijection. Since $B_nY$ is normal in $C_nY$, the displayed
multiplication $\star$ is independent of representatives and makes
this bijection a group isomorphism. The construction commutes with
morphisms of augmented simplicial groups, so the isomorphism is
natural. In particular the coset set is a singleton exactly when
$D_nY=Z_nY$, giving the asserted surjectivity criterion.

All operations in the horn-filling recursion are continuous in the
pro-$p$ case and are morphisms of underlying schemes in the
prounipotent case. Images are closed by compactness for pro-$p$
groups, and by closedness of continuous linear images of linearly
compact spaces after passing to Lie algebras in the prounipotent
case. Transporting the structure from the Moore quotient gives the
stated interpretation in both categories. For pro-$p$ groups the
transported topology also agrees with the coset topology, since the
induced continuous bijection is between compact Hausdorff spaces.
\end{proof}

For $0\leq m\leq n$ write
\[
 D(n,m)=\operatorname{Sur}_{\Delta}([n],[m])
\]
for the set of order-preserving surjections in the simplex category.
Indexing by these maps identifies degeneracy words related by the
simplicial identities.

\begin{definition}\label{def:free}
A \emph{based free simplicial object} is a simplicial object together
with sets $S_m$ of nondegenerate free generators and specified
isomorphisms
\begin{equation}\label{eq:basis}
 X_n\cong F(B_n),\qquad
 B_n=\coprod_{0\leq m\leq n}\ \coprod_{\sigma\in D(n,m)}S_m^{\sigma}.
\end{equation}
Here $S_m^{\sigma}$ is a copy of $S_m$. If
$\alpha_i:[n+1]\twoheadrightarrow[n]$ is the elementary surjection
corresponding to $s_i$, then
\[
 s_i(e,\sigma)=(e,\sigma\alpha_i).
\]
All face and degeneracy maps are morphisms in $\C$. The object is of
\emph{finite type} if each $S_m$ is finite.
\end{definition}

In each degree \eqref{eq:partition} makes the new and degenerate
families free factors. This specified basis structure is stronger than
degreewise freeness.

A \emph{relative free
extension} $A_\bullet\to B_\bullet$ is equipped with sets $T_m$ of new
nondegenerate relative generators and decompositions
\[
 B_n\cong A_n\amalg
 F\left(\coprod_{m\leq n}\ \coprod_{\sigma\in D(n,m)}T_m^\sigma\right).
\]
The faces of a new generator $t\in T_n$ form a compatible boundary
$(d_0t,\ldots,d_nt)\in M_nB$; equivalently, the attaching data define
a morphism from $F(T_n)$ to the matching object of the lower skeleton.

\begin{lemma}\label{lem:extension}
Let $A_\bullet\to B_\bullet$ be a relative free extension, with an
augmentation $B_\bullet\to G$. Let $Y_\bullet\to H$ be a resolution
and $\alpha:G\to H$ a morphism. Every simplicial morphism
$a:A_\bullet\to Y_\bullet$ over $\alpha$ extends to a simplicial
morphism $B_\bullet\to Y_\bullet$ over $\alpha$.
\end{lemma}
\begin{proof}
Set $u_{-1}=\alpha$ and construct $u_n:B_n\to Y_n$ inductively,
extending $a_n$ and commuting with faces and degeneracies.
In degree zero, lift $\alpha\varepsilon_B:B_0\to H$ through
$b_0:Y_0\twoheadrightarrow H$, relative to $a_0$, using
$B_0=A_0\amalg F(T_0)$ and Corollary~\ref{cor:relative}.

Suppose $u_j$ has been constructed for $j<n$. Put
\[
 D_n=\coprod_{m<n}\coprod_{\sigma\in D(n,m)}T_m^\sigma,
 \qquad P_n=A_n\amalg F(D_n),\qquad B_n=P_n\amalg F(T_n).
\]
Define $v_n:P_n\to Y_n$ by
\[
 v_n|_{A_n}=a_n,\qquad
 v_n(t^\sigma)=Y(\sigma)u_m(t)\quad(t\in T_m,\ m<n).
\]
For fixed $n$ there are finitely many surjections $[n]\twoheadrightarrow[m]$.
Each generator family maps continuously, so \eqref{eq:partition} gives
a morphism $v_n$ in $\C$.

To check its faces, factor
$\sigma\delta_i=\iota\rho$, where
$\rho:[n-1]\twoheadrightarrow[q]$ and $\iota:[q]\hookrightarrow[m]$.
Since $q\leq m<n$, the induction hypothesis gives
\begin{align*}
 u_{n-1}d_i(t^\sigma)
 &=u_{n-1}B(\rho)B(\iota)t
  =Y(\rho)u_qB(\iota)t\\
 &=Y(\rho)Y(\iota)u_m(t)
  =d_iY(\sigma)u_m(t)=d_iv_n(t^\sigma).
\end{align*}
On $A_n$ the equality follows from the simpliciality of $a$.
The degeneracy identities follow from composition of the surjections
indexing the basis. Thus $v_n$ has the faces and degeneracies prescribed
by the lower truncation.

The lower-degree maps induce $M_n(u):M_nB\to M_nY$. Set
\[
 c_n=M_n(u)b_n^B:B_n\longrightarrow M_nY.
\]
The face calculation gives $b_n^Yv_n=c_n|_{P_n}$.
Since $Y\to H$ is a resolution, $b_n^Y:Y_n\twoheadrightarrow M_nY$
is an epimorphism. Corollary~\ref{cor:relative} therefore extends $v_n$
to $u_n:B_n\to Y_n$ with $b_n^Yu_n=c_n$. This equation gives the
face identities; agreement with $v_n$ gives the degeneracy identities
and $u_n|_{A_n}=a_n$.

Induction yields the required simplicial morphism. Each lift is a
morphism in $\C$, including when the generator families are infinite.
The resolution hypothesis removes the full-boundary obstruction
identified in Proposition~\ref{prop:keune-simplicial-kernel}.
\end{proof}

\section{The cylinder and the comparison theorem}

For a simplicial object $X$ define its cylinder by
\begin{equation}\label{eq:cylinder}
 C(X)_n=\coprod_{\tau:[n]\to[1]}X_n,
\end{equation}
where $\tau$ is order preserving. There are $n+2$ summands.
For a simplicial operator $\theta:[m]\to[n]$, the map from the
$\tau$-summand is $X(\theta)$ followed by the inclusion of the
$\tau\theta$-summand. This defines a simplicial object.
The two constant maps $[n]\to[1]$ give endpoint inclusions
$i_0,i_1:X\to C(X)$, and folding the summands gives $C(X)\to X$.
If $X$ is augmented to $G$, so is $C(X)$.

\begin{lemma}\label{lem:cylinder}
If $X$ is based free, the endpoint inclusion
$X\amalg X\to C(X)$ is a relative free extension.
\end{lemma}
\begin{proof}
By \eqref{eq:basis} and \eqref{eq:cylinder}, the free generators of
$C(X)_n$ have labels
\[
 (\sigma,e,\tau),\quad
 \sigma\in D(n,m),\quad e\in S_m,\quad
 \tau:[n]\to[1].
\]
A label is degenerate precisely when
the two monotone maps $\sigma$ and $\tau$ have a common repetition: there is an $i$ with
$\sigma(i)=\sigma(i+1)$ and $\tau(i)=\tau(i+1)$.
Collapse all maximal consecutive blocks on which the pair
$(\sigma,\tau)$ is constant. This gives a unique factorization
\[
 (\sigma,\tau)=(\bar\sigma,\bar\tau)\rho,
 \qquad \rho:[n]\twoheadrightarrow[r],
\]
where the reduced pair has no common repetition. The reduced label
is the unique nondegenerate label from which the original arises.
If $\tau$ is constant, its generators belong to the endpoint object,
and the reduced labels are exactly its ordinary nondegenerate ones.
All other reduced labels give the new relative generators.

For fixed $n$ there are only finitely many pairs of monotone maps
involved. Thus the partition into endpoints, new generators and
their degeneracies is a finite partition into copies of the $S_m$.
Each family retains its convergent-generator topology, and
\eqref{eq:partition} gives the required decomposition in $\C$.
Deleting a vertex sends a new generator to degree $n-1$, where its
image is already part of the lower skeleton. This proves the
relative free assertion, including continuity of the attaching maps.
\end{proof}

\begin{theorem}\label{thm:comparison}
Let $X_\bullet\to G$ be an augmented based free simplicial object
in $\C$, and let $Y_\bullet\to H$ be a resolution. Every morphism
$\alpha:G\to H$ extends to a simplicial morphism
$f:X_\bullet\to Y_\bullet$. Any two extensions $f_0,f_1$ are
simplicially homotopic over $\alpha$, with all homotopy components
morphisms in $\C$.
Consequently any two based free resolutions of the same group are
simplicially homotopy equivalent over that group. No finite-rank or
finite-presentation assumption is needed.
\end{theorem}
\begin{proof}
Apply Lemma~\ref{lem:extension} to the relative free extension
$1\to X$ to obtain $f$. For two given extensions, their coproduct
defines $X\amalg X\to Y$ over $\alpha$. By
Lemma~\ref{lem:cylinder} and Lemma~\ref{lem:extension} it extends to
\[
 H:C(X)\longrightarrow Y,\qquad H i_0=f_0,\quad H i_1=f_1.
\]
Write $H_{n,\tau}:X_n\to Y_n$ for the restriction to the
$\tau$-summand. For $0\leq i\leq n$ let
$\tau_i:[n+1]\to[1]$ be $0$ on $0,\ldots,i$ and $1$ on
$i+1,\ldots,n+1$, and put
\[
 h_i^n=H_{n+1,\tau_i}s_i:X_n\longrightarrow Y_{n+1}.
\]
All these maps are morphisms in $\C$. Precomposing the labels $\tau_i$ and using the simplicial identities
in $X$ gives
\begin{align*}
 d_0h_0&=f_1,&d_{n+1}h_n&=f_0,\\
 d_jh_i&=h_{i-1}d_j &&(j<i),\\
 d_jh_i&=h_i d_{j-1} &&(j>i+1),\\
 d_{i+1}h_i&=d_{i+1}h_{i+1} &&(i<n),\\
 s_jh_i&=h_{i+1}s_j &&(j\leq i),\\
 s_jh_i&=h_i s_{j-1} &&(j>i).
\end{align*}
These give a homotopy from $f_0$ to $f_1$ with the indicated endpoint
convention. The augmentation is preserved because $H$ is over
$\alpha$.

For two based free resolutions of $G$, construct maps in both
directions over $\id_G$. Each composite and the relevant identity
are extensions of $\id_G$. The homotopy uniqueness just proved
therefore makes the maps homotopy inverses.
\end{proof}

\section{Applications and the model structure}

\begin{corollary}\label{cor:constant}
Let $P_\bullet\to F(S)$ be a based free resolution in $\C$.
Then $P_\bullet$ is simplicially homotopy equivalent over $F(S)$
to $\Const(F(S))$, by morphisms and homotopies in $\C$.
\end{corollary}
\begin{proof}
The constant object is a based free resolution: take $S_0=S$ and
$S_n=\varnothing$ for $n>0$, and use the identity augmentation.
All its matching maps are isomorphisms. Apply
Theorem~\ref{thm:comparison}.
\end{proof}

\subsection*{The model structure}

Quillen's theorem on simplicial objects, together with its cogroup
criterion \cite[Chapter~II, \S4]{QuillenHA}, applies as follows.
Both categories $\C$ have finite limits. Their effective epimorphisms
are precisely the surjective morphisms used in this paper: images are
closed, and a continuous surjection is the coequalizer of its kernel
pair. In the prounipotent case this assertion follows equally from
the corresponding quotient statement for pronilpotent Lie algebras.

There are sufficiently many projectives. Every pro-$p$ group is a
quotient of a free pro-$p$ group
\cite[Sections~7.6--7.7]{RibesZalesskii}. In characteristic zero, for
$U=\exp(\mathfrak u)$ the identity on the underlying linearly compact
space of $\mathfrak u$ extends to a continuous surjection
\[
 \LL(\mathfrak u)\twoheadrightarrow\mathfrak u.
\]
The source is projective by Proposition~\ref{prop:projective} and its
proof. After choosing a topological vector-space isomorphism
$\mathfrak u\cong k^S$, its exponential is a free object $F_k(S)$.
Thus in both categories every object is a quotient of a free
projective object of the kind used above.

These free objects are also cogroup objects. For the coproduct
inclusions $j_1,j_2:F(S)\to F(S)\amalg F(S)$ and canonical
generators $g_s$, the assignments
\[
 g_s\longmapsto j_1(g_s)j_2(g_s),\qquad
 g_s\longmapsto g_s^{-1}
\]
define the comultiplication and coinverse, with the unique map to $1$
as counit. The assignments are convergent, so the universal property
gives morphisms in $\C$; the cogroup identities follow by checking
them on the generators. Consequently
$\Hom_{\C}(F(S),Y_\bullet)$ is a simplicial group and hence a Kan
complex. Every projective object $Q$ is a retract of a free object,
so $\Hom_{\C}(Q,Y_\bullet)$ is a retract of a Kan complex and is
itself Kan.

This verifies Quillen's condition that every object of $s\C$ is
fibrant. His theorem therefore gives a closed simplicial model
structure in which fibrations and weak equivalences are detected by
$\Hom_{\C}(Q,-)$ for all projective objects $Q$; cofibrations are
the maps with the left lifting property against trivial fibrations.
This argument uses the cogroup criterion, without any assumption
that the free objects are small. The direct proof of
Theorem~\ref{thm:comparison} above is independent of this model
structure.

In this model structure, an augmentation $P_\bullet\to\Const(G)$
whose full matching maps are surjective is a trivial fibration, and
in particular a weak equivalence. Indeed, $\Hom_{\C}(Q,-)$ preserves
the finite limits defining matching objects, and projectivity makes
the induced matching maps surjective. Thus
\[
 \Hom_{\C}(Q,P_\bullet)\longrightarrow
 \Const\bigl(\Hom_{\C}(Q,G)\bigr)
\]
has the right lifting property for every simplicial boundary
inclusion, which is the trivial-fibration criterion for simplicial
sets. Hence a free resolution of an arbitrary $G$ and $\Const(G)$
represent the same object of $\operatorname{Ho}(s\C)$. The
$K(G,0)$ terminology refers to this weak homotopy type as a
simplicial group object; the corresponding classifying construction
gives the $K(G,1)$ interpretation.

Theorem~\ref{thm:comparison} gives simplicial homotopy equivalences
between based free resolutions of the same $G$. If $G=F(S)$,
Corollary~\ref{cor:constant} also gives such an equivalence with
$\Const(F(S))$. For general $G$, a homotopy inverse over $G$ to
$P_\bullet\to\Const(G)$ would require a section of
$P_0\twoheadrightarrow G$, which need not exist.

For the completed presentation model in \cite{Mikhovich},
Corollary~\ref{cor:constant} applies after the resolving property has
been proved. It then gives the comparison with the constant free object
by morphisms and homotopies in the completed category.

The lifting argument does not require a section of a matching map.
With $N_nY=\bigcap_{i=0}^{n-1}\ker d_i$, its kernel is
\[
 \ker b_n=N_nY\cap\ker d_n.
\]
Only the prescribed morphism from a free source is lifted, relative
to the part already defined. Neither finite presentation of $G$ nor
finite type of the resolution is required.

\end{document}